\documentclass[article,12pt]{amsart}
\usepackage{amssymb,amsmath,amscd,enumerate,graphicx,latexsym, amsfonts,verbatim}
\usepackage[all]{xy}
\usepackage{tikz}

\numberwithin{equation}{section}
\newtheorem{theorem}{Theorem}[section]
\newtheorem{lemma}[theorem]{Lemma}
\newtheorem{proposition}[theorem]{Proposition}
\newtheorem{corollary}[theorem]{Corollary}
\newtheorem{conjecture}[theorem]{Conjecture}

\theoremstyle{definition}

\newtheorem{def-prop}[theorem]{Definition-Proposition}
\newtheorem{remark}[theorem]{Remark}
\newtheorem{example}[theorem]{Example}

\newtheorem*{acknowledgement}{Acknowledgement}

\newtheorem{question}[theorem]{Question}
\newtheorem{problem}[theorem]{Problem}
\newtheorem*{Mysketch}{Sketch of proof} 
  {\pushQED{\qed}\begin{Mysketch}}
  {\popQED\end{Mysketch}}

\DeclareMathOperator{\Min}{Min}

\DeclareMathOperator{\height}{ht}

\DeclareMathOperator{\mongrade}{mon-grade}

\DeclareMathOperator{\supp}{supp}
\DeclareMathOperator{\rk}{rk}

\newcommand{\NN}{{\mathbb N}}
\newcommand{\QQ}{{\mathbb Q}}
\newcommand{\RR}{{\mathbb R}}

\def\E{{\mathcal E}}
\def\G{{\mathcal G}}

\def\F{{\mathcal F}}
\def\G{{\mathcal G}}

\def\H{{\mathcal H}}

\def\a{{\bf a}}
\def\b{{\bf b}}

\def\e{{\bf e}}

\def\y{{\bf y}}
\def\z{{\bf z}}

\def\1{{\bf 1}}
\def\0{{\bf 0}}

\begin{document}

\title{Membership criteria and containments of powers of monomial ideals}

\author{Huy T\`ai H\`a}
\address{Tulane University \\ Department of Mathematics \\
6823 St. Charles Ave. \\ New Orleans, LA 70118, USA}
\email{tha@tulane.edu}
\urladdr{http://www.math.tulane.edu/$\sim$tai/}

\author[Trung]{Ngo Viet Trung}
\address{International Centre for Research and Postgraduate Training\\ Institute of Mathematics \\ Vietnam Academy of Science and Technology \\ 18 Hoang Quoc Viet, Hanoi, Vietnam}
\email{nvtrung@math.ac.vn}
\urladdr{}

\begin{abstract}
We present a close relationship between matching number, covering numbers and their fractional versions in combinatorial optimization and ordinary powers, integral closures of powers, and symbolic powers of monomial ideals.
This relationship leads to several new results and problems on the containments between these powers.
\end{abstract}


\maketitle

\centerline{\em \small Dedicated to Le Tuan Hoa on the occasion of his 60th birthday }

\section*{Introduction}

Let $M$ be an $n \times m$ matrix of non-negative integers and $\a \in \NN^n$.
Consider the integer programming problems \par
\begin{enumerate}
\item maximize $\1^m \cdot \y$, \par
\noindent subject to $M \cdot \y \le \a,\ \y \in \NN^m$
\item
minimize $\a \cdot \z$,\par
\noindent subject to $M^T \cdot \z \ge \1^n,\ \z \in \NN^n$\par
\end{enumerate}
where $\1^m = (1,\dots,1) \in \NN^m$ and $\1^n = (1,\dots,1) \in \NN^n$.

Let $\nu_\a(M)$, $\tau_\a(M)$ and $\nu^*_\a(M)$, $\tau_\a^*(M)$ denote the optimal values of these integer programming problems and their fractional relaxations, respectively. Then
$$\nu_\a(M) \le \nu_\a^*(M) = \tau^*_\a(M) \le \tau_\a(M),$$
where the middle equality follows from the duality in linear programming.
The numbers $\nu_\a(M)$, $\tau_\a(M)$, $\nu^*_\a(M)$, $\tau_\a^*(M)$ are important invariants in combinatorial optimization.
For instance, if $M$ is the incidence matrix of a hypergraph $\H$ and $\a = \1^n$,  then $\nu_\a(M)$ and $\tau_\a(M)$ are the matching and covering numbers $\nu(\H)$ and $\tau(\H)$ of $\H$.

Let $I$ be a monomial ideal in a polynomial ring $R$ over a field $K$.
For $k \ge 1$, let $\overline{I^k}$ and $I^{(k)}$ denote the integral closure of $I^k$ and the $k$-th symbolic power of $I$. We may call $\overline{I^k}$ the $k$-th \emph{integral power} of $I$.
Define $\nu_\a(I)$, $\tau_\a(I)$, $\nu^*_\a(I)$, $\tau_\a^*(I)$ to be $\nu_\a(M)$, $\tau_\a(M)$, $\nu^*_\a(M)$, $\tau_\a^*(M)$, respectively, where $M$ is the exponent matrix of the monomial generators of $I$.
The main goal of this paper is to show that the invariants $\nu_\a(I)$, $\tau_\a(I)$, $\nu^*_\a(I)$, and $\tau_\a^*(I)$ can be used to
study the behavior of $I^k$, $\overline{I^k}$ and $I^{(k)}$.

Throughout the paper, $R = K[x_1,\dots,x_n]$ denotes a polynomial ring over a field $K$.
For a vector $\a = (\alpha_1,\dots,\alpha_n) \in \NN^n$, set $x^\a = x_1^{\alpha_1}\cdots x_n^{\alpha_n}$.
Our work hinges on the following effective membership criteria for a monomial $x^\a$ to be in $I^k$, $\overline{I^k}$ and $I^{(k)}$.
\medskip

\noindent{\bf Propositions \ref{ordinary} and \ref{symbolic}.}
Let $I$ be an arbitrary monomial ideal in $R$. Then \par
{\rm (i)} $x^\a \in  I^k$ if and only if $\nu_\a(I) \ge k$,\par
{\rm (ii)} $x^\a \in \overline{I^k}$  if and only if $\nu_\a^*(I) \ge k$.\par
\noindent Moreover, if $I$ is a squarefree monomial ideal, then \par
{\rm (iii)} $x^\a \in I^{(k)}$ if  and only if $\tau_\a(I) \ge k$.
\medskip

If $I$ is squarefree monomial ideal, then  $\nu_\a(I)$, $\tau_\a(I)$, $\nu_\a^*(I)$, and $\tau_\a^*(I)$  are equal to the matching, covering, fractional matching, and fractional covering numbers of a hypergraph. The gaps between these invariants have been studied extensively in combinatorial optimization (see, for example, the survey \cite{Du}).

We shall combine gap estimates between (fractional) matching and covering numbers of hypergraphs  with the membership criteria mentioned above to derive containments between corresponding powers of squarefree monomial ideals. Specifically, by letting $d(I)$ denote the maximum degree of the minimal generators of a monomial ideal $I$, we obtain the following result.
\medskip

\noindent{\bf Theorem \ref{thm.containments}.}
Let $I$ be a squarefree monomial ideal, and let $r = d(I)$. Then, for any $k \ge 1$, we have
\begin{enumerate}
\item[{\rm (i)}] $\overline{I^{(r-1)(k-1) + \big\lceil\frac{k}{r}\big\rceil}} \subseteq I^k$;
\item[{\rm (ii)}] $I^{(\lceil (1+ \frac{1}{2} + \cdots + \frac{1}{r})k\rceil)} \subseteq \overline{I^k}$;
\item[{\rm (iii)}] $I^{(rk-r+1)} \subseteq I^k.$
\end{enumerate}
\medskip

The containments in Theorem \ref{thm.containments} are new, even in the case where $r = 2$, i.e., when $I$ is the edge ideal of a graph.
Furthermore, the containment $I^{(rk-r+1)} \subseteq I^k$ yields a new bound on the {\em resurgence} number of $I$; this invariant, for any homogeneous ideal $I \subseteq R$, was defined by Bocci and Harbourne \cite{BH} to be
$$\rho(I) := \sup\Big\{\frac{h}{k} ~\Big|~ I^{(h)} \not\subseteq I^k\Big\}.$$

\noindent{\bf Corollary \ref{HB}.}
Let $I$ an arbitrary squarefree monomial ideal. Then
$$\rho(I) \le d(I).$$

On the other hand, we shall also use known containments of powers of ideals to provide new estimates for gaps between the invariants
$\nu_\a(I)$, $\nu_\a^*(I)$, $\tau_\a(I)$, and $\tau^*_\a(I)$.
Particularly, thanks to the celebrated Brian\c{c}on-Skoda theorems of Lipman and Sathaye \cite{LS} and Lipman and Teissier \cite{LT} (see also \cite{Huneke}), we have
$\overline{I^{k + \min\{m,n\}-1}} \subseteq I^k$ for all $k \ge 1$, where $m$ is the minimal number of generators of $I$.
Applying the membership criteria for $I^k$ and $\overline{I^k}$,  we obtain the following bound
for the gap between $\nu_\a(M)$ and $\nu^*_\a(M)$, which seems to be unknown in combinatorial optimization.
\medskip

\noindent{\bf Theorem \ref{BS}.}
Let $M$ be an $n \times m$ matrix of non-negative integers. Then for all $\a \in \NN^n$,
$$\nu^*_\a(M) < \nu_\a(M) + \min\{m,n\}.$$

If $I$ is a squarefree monomial ideal, it is known that $I^{(hk-h+1)} \subseteq I^k$ for all $k \ge 1$,
where $h$ is the maximal height of an associated prime of $I$.
This containment gives a positive answer to a conjecture of Harbourne (see \cite{Ba, CEHH}).
We show that it yields the following estimate for the gap between $\tau_\a(M)$ and $\nu_\a(M)$.
\medskip

\noindent{\bf Theorem \ref{Ha}.}
Let $M$ be the incidence matrix of a simple hypergraph $\H$.
Let $h$ be the maximal size of a minimal cover of $\H$.
Then for all $\a \in \NN^n$,
$$\tau_\a(M) \le h\nu_\a(M).$$

One of the famous unsolved problems in combinatorics is Ryser's conjecture \cite{He},
which states that for an $r$-partite $r$-uniform hypergraph $\H$,
$$\tau(\H) \le (r-1)\nu(\H).$$
Using the membership criteria for $I^k$ and $I^{(k)}$,
we can reformulate this conjecture as a problem on the containment between ordinary and symbolic powers of squarefree monomial ideals.
\medskip

\noindent{\bf Conjecture \ref{conj.containment}.}
Let $I$ be the edge ideal of an $r$-partite hypergraph of rank $\le r$. Then for all $k \ge 1$,
$$I^{((r-1)(k-1)+1)} \subseteq I^k.$$

The containment in Conjecture \ref{conj.containment} is true if we replace $I^k$ by $\overline{I^k}$ or $I^{((r-1)(k-1)+1)}$ by $\overline{I^{(r-1)(k-1)+1}}$ (see Theorem \ref{symbolicVSintegral3}).

The membership criteria in Propositions \ref{ordinary} and \ref{symbolic} further allow us to study equalities between $I^k$, $\overline{I^k}$, and $I^{(k)}$, and their combinatorial interpretations.

Following the terminology in combinatorial optimization \cite{Du, Sch}, we say that an $n \times m$ matrix $M$ of non-negative integers has the {\em integer round-down property} if $\nu_\a(M) = \lfloor \nu^*_\a(M) \rfloor$ for all $\a \in \NN^n$. On the other hand, we call a hypergraph $\H$ \emph{Mengerian} (respectively, \emph{K\"onig}) if $\nu_\a(M) = \tau_\a(M)$ for all $\a \in \NN^n$ (respectively, for $\a = \1^n$), where $M$ is the incidence matrix of $\H$. Similarly, we call a hypergraph $\H$ \emph{Fulkersonian} if $\tau^*_\a(M) = \tau_\a(M)$ for all $\a \in \NN^n$. A hypergraph obtained from $\H$ by a sequence of deleting and contracting vertices is called a \emph{minor} of $\H$.

The membership criteria for $I^k$, $\overline{I^k}$ and $I^{(k)}$ immediately yield the following results.
\medskip

\noindent{\bf Theorem \ref{normal}.}  \cite{DV, T2}
Let $I$ be an arbitrary monomial ideal. Then
$\overline{I^k} = I^k$ for all $k \ge 1$ if and only if the exponent matrix of $I$ has the integer round-down property.
\medskip

\noindent{\bf Theorem \ref{correspondence}}  \cite{GVV, HHTZ, T1}
Let $I$ be the edge ideal of a hypergraph $\H$. Then \par
{\rm (i)} $I^{(k)} = I^k$ for all $k \ge 1$ if and only if $\H$ is Mengerian, \par
{\rm (ii)} $I^{(k)} = \overline{I^k}$ for all $k \ge 1$ if and only if $\H$ is Fulkersionian.
\medskip

As an application, we give an algebraic version of the long-standing conjecture of Conforti and Cornu\'ejols, which states that a hypergraph $\H$ is Mengerian if and only if all minors of $\H$ are K\"onig \cite{CC}. For a monomial ideal $I$,  we denote by $\mongrade(I)$ the maximal length of a regular sequence of monomials in $I$.
\medskip

\noindent {\bf Conjecture \ref{CC}.}
Let $I$ be a squarefree monomial ideal such that $\mongrade(J) = \height(J)$
for all monomial ideals $J$ obtained from $I$ by setting some variables equal to 0 or 1.
Then $I$ is a normal ideal.
\medskip

Finally, to give an application of the membership criteria in a topic other than containments between powers of ideals, we study the problem of whether for any squarefree monomial ideals $I$, $d(I^{(k)}) \le kd(I)$ for all $k \ge 1$.
This problem is motivated by a similar question of Huneke \cite{Hu} for homogeneous prime ideals.
We show that this problem is amount to whether $n \le ht(I)d(I)$, where $n$ is the number of variables appearing in the generating monomials of $I$. This leads us to counter-examples to the aforementioned question, in which the difference $d(I^{(k)}) - kd(I)$ can be arbitrarily large.
Other counter-examples were given recently by Asgharzadeh \cite{As} (with an attribute to Hop D. Nguyen).\footnote{Our examples were obtained independently, and that was communicated to Huneke on August 8, 2017.}

The paper is divided into 6 sections.
Section \ref{sec.membership} presents the membership criteria for $I^k$, $\overline{I^k}$, $I^{(k)}$ in terms of the numbers
$\nu_\a(I)$, $\nu_\a^*(I)$, $\tau_\a(I)$, $\tau^*_\a(I)$.
In Section \ref{sec.matchingcovering} we shows that these numbers are the matching and covering numbers of a hypergraph.
Section \ref{sec.gap2containment} is devoted to containments between different powers of $I$, that arise from estimates for the gaps between $\nu_\a(I)$, $\nu_\a^*(I)$, $\tau_\a(I)$, and $\tau^*_\a(I)$.
Section \ref{sec.containment2gap} is to deduce new estimates for the gaps between these numbers from known containments between different powers of $I$.
Section \ref{sec.equality} examines the equalities between $I^k$, $\overline{I^k}$, and $I^{(k)}$.
Section \ref{sec.generatingdegree} deals with the generating degrees of symbolic powers and the aforementioned question of Huneke.

\begin{acknowledgement}
This paper started during a research stay of the authors at Vietnam Institute for Advanced Study in Mathematics. The authors would like to thank the institute for its support and hospitality.
The first author is partially supported by Simons Foundation (grant \# 279786) and Louisiana Board of Regents (grant \# LEQSF(2017-19)-ENH-TR-25). The second author is supported by Vietnam National Foundation for Science and Technology Development (grant \# 101.04-2017.19).
\end{acknowledgement}


\section{Membership problems for powers of monomial ideals} \label{sec.membership}

Let $I$ be a monomial ideal in $R = K[x_1,\dots,x_n]$ and let $x^{\a_1},\dots,x^{\a_m}$ be the minimal monomial generators of $I$.
We call the matrix $M$, whose columns are the vectors $\a_1,\dots,\a_m$, the \emph{exponent matrix} of $I$.
By definition,
\begin{align*}
\nu_\a(I) & := \max\{\1^m \cdot \y ~|~ \y \in \NN^m, M \cdot \y \leqslant \a\},\\
\nu_\a^*(I) & := \max\{\1^m \cdot \y ~|~ \y \in \RR_{\ge 0}^m, M \cdot \y \leqslant \a\},\\
\tau_\a(I) & := \min\{\a\cdot \z ~|~ \z \in \NN^n, M^T \cdot \z \ge \1^m\},\\
\tau_\a^*(I) & := \min\{\a\cdot \z ~|~ \z \in \RR_{\ge 0}^n, M^T \cdot \z \ge \1^m\},
\end{align*}
where $\NN$ denotes the set of natural numbers, including 0, and $\RR_{\ge 0}$ is the set of non-negative real numbers.

The aim of this section is to give effective conditions for a monomial $x^\a$ to be an element of $I^k$, $\overline{I^k}$, or $I^{(k)}$, $k \ge 1$, in terms of the aforementioned invariants associated to $I$ and $\a$. These criteria were already presented without proofs in the lecture note \cite{T2}.

\begin{proposition} \label{ordinary} \cite[Proposition 3.1]{T2}
Let $I$ be an arbitrary monomial ideal. Then \par
{\rm (i)} $x^\a \in  I^k$ if and only if $\nu_\a(I) \ge k$,\par
{\rm (ii)} $x^\a \in \overline{I^k}$  if and only if $\nu_\a^*(I) \ge k$.\par
\end{proposition}

\begin{proof}
(i) It is clear that $x^\a \in  I^k$ if and only if $x^\a$ is divisible by a monomial of the form
$(x^{\a_1})^{\beta_1}\cdots (x^{\a_m})^{\beta_m}$ with $\beta_1 + \cdots + \beta_m \ge  k$.
The divisibility means that $\beta_1\a_1 + \cdots +  \beta_m\a_m \le \a$.
Set $\y = (\beta_1,\dots,\beta_m)$.
Then $\beta_1 + \cdots + \beta_m = \1^m \cdot \y$ and $\beta_1\a_1 + \cdots +  \beta_m\a_m = M \cdot \y$.
From this observation we can conclude that $x^\a \in  I^k$ if and only if $\nu_\a(I) \ge k$. \par

(ii) It is well-known that $x^\a \in \overline{I^k}$ if and only if there is an integer $q \ge 1$ such that $x^{q\a} \in I^{q k}$. By (i), this means that $\nu_{q \a}(I) \ge q k$.
This condition implies the existence of $\y \in \NN^m$ such that $\1^m \cdot \y \ge q k$ and $M \cdot \y \le q \a$.
Since $\frac{1}{q}\y \cdot \1^m \ge k$ and $M \cdot \frac{1}{q}\y \le \a$, we obtain $\nu_\a^*(I) \ge k$.\par

Conversely, if $\nu_\a^*(I) \ge k$, then there exists $\y' \in \RR_{\ge 0}^n$ such that $\1^m \cdot \y' \ge k$ and $M \cdot \y' \le \a$.
Since $M$ is a matrix of integers and $\a \in \NN^n$, we may choose $\y'$ to be a rational vector. Then $\y' = \frac{1}{q}\y$ for some $\y \in \NN^m$ and $q \in \NN$. Since $\y \cdot \1^m \ge q k$ and $M \cdot \y \le q \a$, we obtain $\nu_{q \a}(I) \ge q k$. That is, $x^{q\a} \in I^{qk}$ and, so, $x^\a \in \overline{I^k}$. Hence, we can conclude that $x^\a \in \overline{I^k}$ if and only if $\nu_\a^*(I) \ge k$.
\end{proof}

\begin{remark} \label{nu*}
The proof of Proposition \ref{ordinary}(ii) shows that
$$\nu_\a^*(I) = \max_{q \ge 1} \frac{\nu_{q\a}(I)}{q}.$$
\end{remark}

To present an effective criterion for $x^\a \in I^{(k)}$ we first need to know the minimal primes of $I$.
Let $\Min(I)$ denote the set of minimal associated primes of $I$.
For every prime ideal $P \in \Min(I)$, there is a subset $F \subseteq [1,n]$ such that $P = P_F$,
where $P_F$ denotes the ideal generated by the variables $x_i$, $i \in F$.
We denote by $I_P$ the $P$-primary component of $I$.

\begin{proposition} \label{symbolicpower}
Let $I$ be an arbitrary monomial ideal.
Then $x^\a \in I^{(k)}$ if and only if $\nu_\a(I_P) \ge k$ for all $P \in \Min(I)$.
\end{proposition}

\begin{proof}
By \cite[Lemma 3.1]{HHT}, we have
$$I^{(k)} = \bigcap_{P \in \Min(I)} I_P^k.$$
Therefore, $x^\a \in I^{(k)}$ if and only if $x^\a \in I_P^k$ for all $P \in \Min(I)$.
By Proposition \ref{ordinary}(i), this condition means that $\nu_\a(I_P) \ge k$ for all $P \in \Min(I)$.
\end{proof}

If $I$ is a squarefree monomial ideal then we have a simpler criterion for $x^\a \in I^{(k)}$.
Before stating this criterion, we shall recall some basic fact from hypergraph theory.

Recall that a \emph{hypergraph} $\H$ consists of a vertex set and a collection of nonempty subsets of the vertex set.
These subsets are called edges (or hyperedges) of $\H$. Graphs are hypergraphs whose edges have size 2.
A hypergraph is \emph{simple} (or a \emph{clutter}) if there are no nontrivial inclusion among the edges.

Unless otherwise specified, \emph{we shall always assume that $\H$ is a simple hypergraph on the vertex set $[1,n] = \{1, \dots, n\}$}. \par

For every subset $F \subseteq [1,n]$ we denote by $\e_F$ the {\em incidence vector} of $F$,
whose $i$-th coordinate equals 1 if $i \in F$ and 0 if $i \not\in F$.
To every hypergraph $\H$ one can assign a squarefree monomial ideal
which is generated by the monomials $x^{\e_F}$, where $F$ is an edge in $\H$.
This ideal is called the {\em edge ideal} of $\H$, and denoted by $I(\H)$.
It is clear that every squarefree monomial ideal can be viewed as the edge ideal of a hypergraph. \par

A subset $F \subseteq [1,n]$ is called a (vertex) {\em cover} or {\em blocking set} of $\H$ if $F$ meets every edge of $\H$.
We denote by $\H^\vee$ the hypergraph whose edges are minimal vertex covers of $\H$. This is also a simple hypergraph, called the {\em blocker} of $\H$. Note that $(\H^\vee)^\vee = \H$.

\begin{lemma} \label{blocker}
Let $I$ be the edge ideal of a hypergraph $\H$. Then
$$\tau_\a(I) = \min\{\a \cdot \e_F|\ F \in \H^\vee\}.$$
\end{lemma}

\begin{proof} Let $x^{\a_1}, \dots, x^{\a_m}$ be the minimal monomial generators of $I$, and let $M$ be the exponent matrix of $I$.
Note that $F$ is a cover of $\H$ if and only $\a_i \cdot \e_F \ge 1$ for all $i = 1,\dots,m$.
This condition can be rewritten as $M^T \cdot \e_F \ge \1^m$.
It is clear that an optimal solution to the linear program of minimizing $\a \cdot \z$ subject to $M^T \cdot \z \ge \1^m$, $\z \in \NN^n$,
can be chosen to be a 0-1 vector $\z$ such that $\supp(\z)$ is of minimal size.
Therefore, such a vector $\z$ must be the incidence vector $\e_F$ of a minimal cover $F$ of $\H$.
It then follows that
$$\tau_\a(I) = \min\{\a \cdot \z ~|~ M^T \cdot \z \ge \1^m,\ \z \in \NN^n\} = \min\{\a \cdot \e_F ~|~ F \in \H^\vee\}.$$
\end{proof}

\begin{proposition} \label{symbolic} \cite[Lemma 3.5(3)]{T2}
Let $I$ be a squarefree monomial ideal. Then
$x^\a \in I^{(k)}$ if and only if $\tau_\a(I) \ge k$.
\end{proposition}

\begin{proof}
Since $I$ is a squarfree monomial ideals, we may consider $I$ as the edge ideal of a hypergraph $\H$.
It is easy to see that $P_F$ is a minimal prime of $I$ if and only if $F$ is a minimal cover of $\H$.
Therefore,  $I = \bigcap_{F \in \H^\vee} P_F.$ This implies that
$$I^{(k)} = \bigcap_{F \in \H^\vee} P_F^k.$$
Thus, $x^\a \in I^{(k)}$ if and only if $x^\a \in P_F^k$ for all $F \in \H^\vee$.
We have $x^\a \in P_F^k$ if and only if $\e_F \cdot \a \ge k$.
Hence, $x^\a \in I^{(k)}$ if and only if $\min\{\e_F \cdot \a ~|\ F \in \H^\vee\} \ge k.$
The conclusion follows by applying Lemma \ref{blocker}.
\end{proof}


\section{Matching and covering numbers of hypergraphs} \label{sec.matchingcovering}

Let $\H$ be a hypergraph. A family of disjoint edges is called a {\em matching} of $\H$.
The minimal size of a maximal matching of $\H$ is called the {\em matching number} of $\H$, denoted by $\nu(\H)$.
The maximal size of a cover of $\H$ is called the {\em covering number} of $\H$, denoted by $\tau(\H)$.

Let $M$ be the {\em incidence matrix} of $\H$ whose columns are the incidence vectors of the edges of $\H$.
It is well-known that
\begin{align*}
\nu(\H) & = \max\{\y \cdot \1^m ~|~  M \cdot \y  \le \1^n, \y \in \NN^m\},\\
\tau(\H) & = \min\{\z \cdot \1^n ~|~  M^T \cdot \z  \ge \1^m, \z \in \NN^n\}.
\end{align*}
The following numbers are called the {\em fractional matching number} or the {\em fractional covering number} of $\H$:
\begin{align*}
\nu^*(\H) & := \max\{\1^m \cdot \y ~|~  M \cdot \y  \le \1^n, \y \in \RR_{\ge 0}^m\},\\
\tau^*(\H) & := \min\{\1^n \cdot \z ~|~  M^T \cdot \z  \ge \1^m, \z \in \RR_{\ge 0}^n\}.
\end{align*}

In this section, we shall see that if $I$ is the edge ideal of a hypergraph $\H$ then the invariants $\nu_\a(I), \tau_\a(I)$, $\nu_\a^*(I)$, and $\tau_\a^*(I)$ can be viewed as the matching number, the covering number, and their fractional versions of a hypergraph associated to $\H$ and $\a$.
Specifically, let $\H^\a$ denote the hypergraph on the vertex set
$$V = \{(i,j)|\ i = 1,\dots,n,\ j = 1,\dots,\alpha_i \},$$
 whose edges are subsets of $V$ of the form $\{(i_1,j_1),\dots,(i_s,j_s)\}$,
where $\{i_1,\dots,i_s\}$ is an edge of $\H$ and $j_1 = 1,\dots,\alpha_{i_1},\dots, j_s = 1,\dots,\alpha_{i_s}$.
The hypergraph $\H^\a$ is called the {\em parallelization} of $\H$ with respect to $\a$.
Note that $\H = \H^\a$ if $\a = \1^n$.

\begin{example} 
Figure \ref{fig1} depicts a hypergraph $\H$ and its parallelization $\H^\a$ with $\a = (1,1,2,2)$.

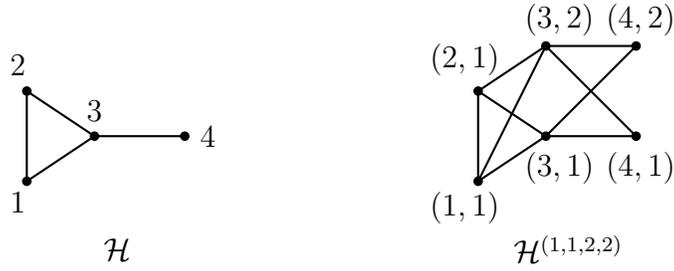
\begin{figure}[ht!] 
\begin{tikzpicture}[scale=0.6]

\draw [thick] (0,0) coordinate (a) -- (0,2) coordinate (b) ;
\draw [thick] (0,2) coordinate (b) -- (1.5,1) coordinate (c) ;
\draw [thick] (1.5,1) coordinate (c) -- (0,0) coordinate (a) ;
\draw [thick] (1.5,1) coordinate (c) -- (3.5,1) coordinate (d);

\draw (2,-1.5) node{$\H$};

\draw (3.6,1) node[right] {$4$};
\draw (-0.2,0) node[left, below] {$1$};
\draw (-0.2,2.1) node[left, above] {$2$};
\draw (1.5,1.1) node[above] {$3$};
\fill (a) circle (3pt);
  \fill (b) circle (3pt);
  \fill (c) circle (3pt);
  \fill (d) circle (3pt);

\draw [thick] (10,0) coordinate (a) -- (10,2) coordinate (b) ;
\draw [thick] (10,2) coordinate (b) -- (11.5,1) coordinate (c) ;
\draw [thick] (11.5,1) coordinate (c) -- (10,0) coordinate (a) ;
\draw [thick] (11.5,1) coordinate (c) -- (13.5,1) coordinate (d);
\draw [thick] (10,2) coordinate (b) -- (11.5,3) coordinate (e) ;
\draw [thick] (10,0) coordinate (a) -- (11.5,3) coordinate (e) ;
\draw [thick] (11.5,3) coordinate (e) -- (13.5,1) coordinate (d);
\draw [thick] (11.5,3) coordinate (e) -- (13.5,3) coordinate (f);
\draw [thick] (11.5,1) coordinate (c) -- (13.5,3) coordinate (f);

\draw (12,-1.5) node{$\H^{(1,1,2,2)}$};

\draw (13.6,0.9) node[right, below] {$(4,1)$};
\draw (9.7,0) node[left, below] {$(1,1)$};
\draw (9.7,2.1) node[left, above] {$(2,1)$};
\draw (11.8,0.9) node[right,below] {$(3,1)$};
\draw (11.8,3) node[right,above] {$(3,2)$};
\draw (13.6,3) node[right,above] {$(4,2)$};
\fill (a) circle (3pt);
  \fill (b) circle (3pt);
  \fill (c) circle (3pt);
  \fill (d) circle (3pt);
  \fill (e) circle (3pt);
   \fill (f) circle (3pt);

\end{tikzpicture}
\caption{Parallelization of a hypergraph.}\label{fig1}
\end{figure}
\end{example}

For every set $E \subseteq V$ we define $p(E) :=  \{i|\ \text{there is $j$ such that $(i,j) \in E$}\}$.
In other words, $p$ is the projection to the first component of the elements of $V$. 
Let $A = \supp(\a)$.
Then $p$ gives a map from $\H^\a$ to $\H_A$,
where $\H_A$ denotes the hypergraph on the vertex set $A$ which consists of edges $F \subseteq A$ of $\H$. \par

The maximal matchings and minimal covers of $\H^\a$ can be described in terms of $\H$ and $\a$ as follows.

\begin{lemma} \label{parallel-1}
Let $\H^\a$ be a parallelization of a hypergraph $\H$. Then \par
{\rm (i)} A family of disjoint edges $E_1,\dots,E_s$ of $\H^\a$ is a maximal matching of $\H^\a$ if and only if $p(E_1),\dots,p(E_s)$ is maximal among sequences $\F$ of (not necessarily distinct) edges of $\H$ with the property $|\{F \in \F|\ i \in F\}| \le a_i$ for all $i = 1,\dots,n$.\par
{\rm (ii)}  A set $C \subseteq V$ is a minimal cover of $\H^\a$ if and only if $C = p^{-1}(D)$ for a minimal cover $D$ of $\H_A$.
\end{lemma}

\begin{proof}
(i) If $E_1,\dots,E_s$ is not a maximal matching, then there is a larger matching $E_1,\dots,E_{s+1}$.
Let $\F$ be the family $p(E_1),\dots,p(E_{s+1})$.
Since $E_1,\dots,E_{s+1}$ are disjoint, we have
$|\{F \in \F|\ i \in F\}| \le |p^{-1}(i)| = a_i$
for all $i = 1,\dots,n$. \par

Conversely, if $p(E_1),\dots,p(E_s)$ is not maximal among sequences $\F$ of not necessarily distinct edges of $\H$
with the property $|\{F \in \F|\ i \in F\}| \le a_i$ for all $i = 1,\dots,n$, 
we put $F_j = p(E_j)$, $j = 1,\dots,s$.
Then there exists an edge $F_{s+1}$ of $\H$ such that the family $\F = \{F_1,\dots,F_{s+1}\}$ satisfies the property
$|\{F \in \F|\ i \in F\}| \le a_i$.
By the definition of $\H^\a$, we can find an edge $E_{s+1} \in p^{-1}(F_{s+1})$ disjoint from the edges $E_1,\dots,E_s$.
Hence, $E_1,\dots,E_s$ is not a maximal matching. \par

(ii) Let $C$ be a minimal cover of $\H^\a$ and let $D = p(C)$. Let $i \in D$ and let $(i,j)$ be any vertex in $C$. Since $C \setminus\{(i,j)\}$ is not a cover of $\H^\a$, there exists an edge $E \in \H^\a$ such that $E \cap C = \{(i,j)\}$. By considering edges in $\H^\a$ obtained from $E$ by replacing $(i,j)$ with $(i,j')$, for $1 \le j' \not= j \le a_i$, it follows that $C$ contains all vertices $(i,j')$ for $1 \le j' \le a_i$. Thus, $C = p^{-1}(D)$.
Since $C$ is a cover of $\H^\a$, $D = p(C)$ is a cover of $\H_A = p(\H^\a)$. 
If there exists $i \in D$ such that $D \setminus \{i\}$ is also a cover of $\H_A$, 
then $p^{-1}(D \setminus \{i\})$ is a cover of $\H^\a$ that is strictly contained in $C$, a contradiction. 
Thus, $D$ is a minimal cover of $\H_A$.

Conversely, let $C = p^{-1}(D)$ for a minimal cover $D$ of $\H_A$. Consider any edge $E \in \H^\a$. Then $p(E)$ is an edge in $\H_A$, and so $p(E) \cap D \not= \emptyset$. This implies that $E \cap C \not= \emptyset$. Thus, $C$ is a cover of $\H^\a$.
Suppose that $C$ is not a minimal cover of $\H^\a$. That is, there exists $i \in D$ and $1 \le j \le a_i$ (particularly, we have $a_i \ge 1$) such that $C \setminus \{(i,j)\}$ is a cover of $\H^\a$. Since $D$ is a minimal cover of $\H_A$, there exists an edge $F \in \H_A$ such that $F \cap D = \{i\}$. Let $E$ be an edge in $\H^\a$ with $p(E) = F$ and $(i,j) \in E$. Then, clearly, $E \cap C = \emptyset$, a contradiction. Hence, $C$ is a minimal cover of $\H^\a$.
\end{proof}

\begin{proposition} \label{parallel-2}
Let $I$ be the edge ideal of a hypergraph $\H$. Then
\begin{enumerate}
\item[\rm (i)] $\nu_\a(I) = \nu(\H^\a)$,
\item[\rm (ii)] $\nu_\a^*(I) = \nu^*(\H^\a)$,
\item[\rm (iii)] $\tau_\a^*(I) = \tau^*(\H^\a)$,
\item[\rm (iv)] $\tau_\a(I) = \tau(\H^\a)$.
\end{enumerate}
\end{proposition}

\begin{proof}
(i) Assume that the edges in $\H$ are $\{F_1,\dots,F_m\}$. We may represent any sequence $\F$ of not necessarily distinct edges of $\H$ as a vector
$\y = (\beta_1,\dots,\beta_m) \in \NN^m$ such that for $j = 1,\dots,m$, $\beta_j$ is the number of times $F_j$ appears in $\F$.
Let $M$ be the exponent matrix of $I$. Then $M$ is an $n \times m$ matrix whose columns are the incidence vectors of $F_1,\dots,F_m$.
Thus, it can be seen that $|\{F \in \F|\ i \in F\}| \le a_i$ for all $i = 1,\dots,n$ if and only if $M \cdot \y \le \a$.
By Lemma \ref{parallel-1}(i), we have
$$\nu(\H^\a) = \max\{\1^m \cdot \y|\ \y \in \NN^m, M \cdot \y \le \a\} = \nu_\a(I).$$ \par

(ii) Let $I_\a$ denote the edge ideal of $\H^\a$. Then $\nu^*(\H^\a) = \nu_{\1^s}^*(I_\a)$, where $s$ is the number of vertices of $\H^\a$. For every integer $q \ge 1$, we can interpret $\nu_{q\1^s}(I_\a)$ as the maximal size of a family $\E$ of not necessarily distinct edges of $\H^\a$ such that every vertex of $V$ appears at most $q$ times in the edges of $\E$. It follows, by a similar argument to that of part (i), that
$$\nu_{q\1^s}(I_\a) = \max\{\1^m \cdot \y|\ \y \in \NN^m, M \cdot \y \le q\a\} = \nu_{q\a}(I).$$
By Remark \ref{nu*}, we have
$$\nu^*(\H^\a) = \max_{q\ge 1} \frac{\nu_{q\1^s}(I_\a)}{q} = \max_{q\ge 1} \frac{\nu_{q\a}(I)}{q} = \nu_\a^*(I).$$
\par

(iii) follows from (i) because $\nu_\a^*(I) = \tau_\a^*(\H^\a)$ and $\nu^*(\H^\a) = \tau^*(\H^\a)$ by the duality of linear programming.\par

(iv) Let $s$ be the number of vertices of $\H^\a$. Using Lemma \ref{blocker}, we have
$$\tau(\H^\a) = \min\{\1^s \cdot \e_E|\ E \in (\H^\a)^\vee\}.$$
By Lemma \ref{parallel-1}(ii), $(\H^\a)^\vee = \{p^{-1}(F)|\ F \in (\H_A)^\vee\}$.
If $E = p^{-1}(F)$ then we have $\1^s \cdot \e_E  = \a \cdot \e_F$.
Therefore,
$$\tau(\H^\a) = \min\{\a \cdot \e_F|\ F \in (\H_A)^\vee\}.$$
By Lemma  \ref{blocker}, we also have
$$\tau_\a(I) = \min\{\a \cdot \e_G|\ G \in \H^\vee\}.$$

For every minimal cover $F$ of $\H_A$,
we consider the hypergraph $\H'$ of the edges of $\H$ not meeting $F$.
Since $[1,n] \setminus A$ is a cover of $\H'$,
there is a minimal cover $F'$ of $\H'$ in $[1,n] \setminus A$.
It is easy to check that $G = F \cup F'$ is a minimal cover of $\H$ with $G \cap A = F$.
Since $\a \cdot \e_F = \a \cdot \e_G$, we get $\tau(\H^\a) \ge \tau_\a(I).$

On the other hand, for every minimal cover $G$ of $\H$,
$F = G \cap A$ is a minimal cover of $\H_A$ and $\a \cdot \e_G = \a \cdot \e_F$.
Therefore, $\tau_\a(I) \ge \tau(\H^\a).$
Hence, $\tau(\H^\a) =  \tau_\a(I).$
\end{proof}


\section{From gap estimates to containments between of ideals} \label{sec.gap2containment}

In general, we have the following correspondence between containments of monomial ideals and bounds on invariants for  membership criteria. This correspondence applies directly to the containments between powers, integral powers, and symbolic powers of a monomial ideal.

\begin{lemma} \label{lem.equiv}
Let $\{I_k\}_{k \ge 1}$ and $\{J_k\}_{k \ge 1}$ be two filtrations of monomial ideals in $R$.
Suppose that there are functions $\mu$ and $\rho$ from $\NN^n$ to $\RR_+$ such that, for any $\a \in \NN^n$ and $k \ge 1$,
\begin{itemize}
\item $x^\a \in I_k$ if and only if $\mu(\a) \ge k$;
\item $x^\a \in J_k$ if and only if $\rho(\a) \ge k$.
\end{itemize}
\noindent Let $f: \NN \longrightarrow \RR_+$ be a non-decreasing function. Then
\begin{enumerate}
\item[{\rm (i)}] $I_k \subseteq J_{\lfloor f(k)\rfloor}$ for all $k \ge 1$ if and only if
$\rho(\a) \ge \lfloor f(\lfloor\mu(\a)\rfloor)\rfloor $ for all $\a \in \NN^n$;
\item[{\rm (ii)}] $I_{\lceil f(k)\rceil} \subseteq J_k$ for all $k \ge 1$ if and only if
$\mu(\a) < \lceil f(\lfloor\rho(\a)\rfloor+1)\rceil$ for all $\a \in \NN^n$.
\end{enumerate}
\end{lemma}

\begin{proof}
(i) Assume that $\rho(\a) \ge \lfloor f(\lfloor \mu(\a)\rfloor)\rfloor$ for all $\a \in \NN^n$.
For an arbitrary monomial $x^\a \in I_k$, we have $\lfloor \mu(\a) \rfloor \ge k$.
Hence, $f(\lfloor \mu(\a)\rfloor) \ge  f(k)$, which implies that $\rho(\a) \ge \lfloor f(k)\rfloor$.
Therefore, $x^\a \in J_{\lfloor f(k)\rfloor}$. Conversely, assume that $I_k \subseteq J_{\lfloor f(k)\rfloor}$ for all $k \ge 1$.
Consider an arbitrary $\a \in \NN^n$, and set $k = \lfloor \mu(\a)\rfloor $.
Then $x^\a \in I_k$. Hence, $x^\a \in J_{\lfloor f(k)\rfloor}$, which implies that $\rho(\a) \ge \lfloor f(\lfloor \mu(\a)\rfloor)\rfloor$.

(ii) Assume that $\mu(\a) < \lceil f(\lfloor \rho(\a)\rfloor +1)\rceil$ for all $\a \in \NN^n$. For
an arbitrary monomial $x^\a \in I_{\lceil f(k)\rceil}$, we have $\mu(\a) \ge \lceil f(k) \rceil$.
Thus, $\lceil f(\lfloor \rho(\a)\rfloor+1) \rceil > \lceil f(k) \rceil$. This implies that $f(\lfloor\rho(\a)\rfloor+1)  >  f(k)$.
Therefore, $\lfloor \rho(\a)\rfloor + 1 > k$. Hence, $\rho(\a) \ge k$ and $x^\a \in J_k$.
Conversely, assume that $I_{\lceil f(k)\rceil} \subseteq J_k$ for all $k \ge 1$. If there exists $\a \in \NN^n$ such that $\mu(\a) \ge \lceil f(\lfloor \rho(\a)\rfloor+1) \rceil$
then $x^\a \in I_{\lceil f(\lfloor \rho(\a)\rfloor+1)\rceil}$ and, since $\lfloor \rho(\a)\rfloor+1 > \rho(\a)$, $x^\a \not\in J_{\lfloor \rho(\a)\rfloor +1}$. This implies that $I_{\lceil f(k)\rceil} \not\subseteq J_k$ with $k = \lfloor \rho(\a)\rfloor+1$, a contradiction.
\end{proof}

In practice, Lemma \ref{lem.equiv}(ii) often is less applicable than the following weaker version, especially when only one direction of the implication is of interest.

\begin{corollary} \label{thm.SymbInt}
Let $\{I_k\}_{k \ge 1}$, $\{J_k\}_{k \ge 1}$, $\mu$, and $\rho$ be as in Lemma \ref{lem.equiv}.
Let $f: \NN \longrightarrow \RR_+$ be a strictly increasing function.
Then $I_{\lceil f(k)\rceil} \subseteq J_k$ for all $k \ge 1$ if $\mu(\a) \le  f(\rho(\a))$ for all $\a \in \NN^n$.
\end{corollary}

\begin{proof}
If $\mu(\a) \le  f(\rho(\a))$, then $\mu(\a) <  f(\lfloor \rho(\a)\rfloor+1)$ because $\rho(\a) < \lfloor \rho(\a)\rfloor+1$ and $f$ is strictly increasing.
Therefore, $\mu(\a) < \lceil f(\lfloor \rho(\a)\rfloor+1)\rceil$ and the conclusion follows from Lemma \ref{lem.equiv}.
\end{proof}

Let $\H$ be a hypergraph.
It follows from the definition of matching and covering numbers (see Section \ref{sec.matchingcovering}) and the duality in linear programming that
$$\nu(\H) \le \nu^*(\H) = \tau^*(\H) \le \tau(\H).$$
The gaps between these invariants has been a major research topic in hypergraph theory (cf. \cite{Du,Sch}).
Estimates for these gaps are often given as bounds for one invariant by a function of another.

If $I$ is the edge ideal of $\H$ then, in light of Proposition \ref{parallel-2}, applying such bounds to the parallelization $\H^\a$ of $\H$, for $\a \in \NN^n$, yields bounds on the invariants $\tau_\a(I)$, $\tau^*_\a(I) = \nu^*_\a(I)$, and $\nu_\a(I)$ for all $\a \in \NN^n$. As we have seen in Section \ref{sec.membership}, these invariants determine whether $x^\a$ belongs to the ideals $I^k$, $\overline{I^k},$ and $I^{(k)}$. Therefore, Lemma \ref{lem.equiv} allows us to derive new containments between these ideals from known bounds on the matching, covering and fractional matching (covering) numbers of hypergraphs.

We will apply this method only to those bounds on $\nu(\H)$, $\nu^*(\H) = \tau^*(\H)$ and $\tau(\H)$,
which involve the \emph{rank} of $\H$. Recall that the \emph{rank} of $\H$, denoted by $\rk(\H)$, is the maximum cardinality of an edge in $\H$.
By the definition of parallelization, $\rk(\H^\a) \le \rk(\H)$ for all $\a \in \NN^n$. Therefore, we would get bounds on the invariants $\tau_\a(I)$, $\tau^*_\a(I) = \nu^*_\a(I)$, and $\nu_\a(I)$, which also involve $\rk(\H)$. On the other hand,
$\rk(\H)$ is just the {\em maximal generating degree} $d(I)$, which denotes
the maximum degree of a minimal monomial generator of $I$.

\begin{theorem} \label{thm.containments}
Let $I$ be a squarefree monomial ideal, and let $r = d(I)$. Then, for any $k \ge 1$, we have
\begin{enumerate}
\item[{\rm (i)}] $\overline{I^{(r-1)(k-1) + \big\lceil\frac{k}{r}\big\rceil}} \subseteq I^k$;
\item[{\rm (ii)}] $I^{(\lceil (1+ \frac{1}{2} + \cdots + \frac{1}{r})k\rceil)} \subseteq \overline{I^k}$;
\item[{\rm (iii)}] $I^{(rk-r+1)} \subseteq I^k.$
\end{enumerate}
\end{theorem}

\begin{proof}
Let $\H$ be the hypergraph associated to $I$.
Note that $\rk(\H^\a) \le \rk(\H) = r$ for all $\a \in \NN^n$. \par

(i) Let $f: \NN \longrightarrow \RR_+$ be the function defined by
$$f(k) := (r-1)(k-1) + \dfrac{k}{r} = \dfrac{r^2-r+1}{r}(k-1) + \dfrac{1}{r}.$$
Clearly, $f$ is a non-decreasing function.
It follows from \cite[Theorem 1.2]{FKS} that
$$\nu^*(\H^\a) \le \frac{r^2-r+1}{r}\nu(\H^\a).$$
This, together with Proposition \ref{parallel-2}, implies that for all $\a \in \NN^n$, we have
 $$\nu^*_\a(I) \le \dfrac{r^2-r+1}{r}\nu_\a(I) < \dfrac{r^2-r+1}{r}\nu_\a(I)+\dfrac{1}{r} = f(\nu_\a(I)+1).$$
Thus, (i) follows by invoking Lemma \ref{lem.equiv}(ii). \par

(ii) Let $f: \NN \longrightarrow \RR_+$ be the function defined by
$$f(k) := \big(1+ \frac{1}{2} + \cdots + \frac{1}{r}\big)k.$$
Then $f(k)$ is a strictly increasing function. By \cite[Proof of Lemma 1.6.4]{SU}, we have
$$\tau(\H^\a) \le \big(1+ \frac{1}{2} + \cdots + \frac{1}{r}\big)\tau^*(\H^\a).$$
for all $\a \in \NN^n$. Proposition \ref{parallel-2} now implies that
$$\tau_\a(I) \le \big(1+ \frac{1}{2} + \cdots + \frac{1}{r}\big)\nu^*_\a(I) = f(\nu^*_\a(I)).$$
Hence, (ii) follows from Corollary \ref{thm.SymbInt}. \par

(iii) It is a basic fact (and easy to see) that $\tau(\H) \le r\nu(\H)$.
Applying this to the parallelization $\H^\a$ we obtain
$$\tau(\H^\a) \le r\nu(\H^\a).$$
By Proposition \ref{parallel-2}, it follows that for all $\a \in \NN^n$,
\begin{align*}
\tau_\a(I) \le r\nu_\a(I) < r(\nu_\a(I)+1) - r +1.
\end{align*}
\noindent Let $f: \NN \longrightarrow \RR_+$ be the function defined by
$$f(k) :=  rk-r+1.$$
Then $f$ is a non-decreasing function and
$$\tau_\a(I) < f(\nu(\a)+1).$$
Hence, (iii) is a consequence of  Lemma \ref{lem.equiv}(ii).
\end{proof}

Even for edge ideals of graphs, Theorem \ref{thm.containments} appears to be new and interesting.

\begin{corollary} \label{cor.graphcontainment}
Let $I$ be the edge ideal of a graph. Then for any $k \in \NN$, we have
\begin{enumerate}
\item[\rm (i)] $\overline{I^{\lceil\frac{3}{2}k\rceil-1}} \subseteq I^k$; 
\item[\rm (ii)]  $I^{(\lceil \frac{3}{2}k\rceil)} \subseteq \overline{I^k}$;
\item[\rm (iii)] $I^{(2k-1)} \subseteq I^k$.
\end{enumerate}
\end{corollary}

For instance, Corollary \ref{cor.graphcontainment}(i) implies a surprising fact that $\overline{I^2} = I^2$,
i.e., $I^2$ is integrally closed.

\begin{example}  The containments in Corollary \ref{cor.graphcontainment} and, thus, in Theorem \ref{thm.containments} are sharp as seen from the following examples. Let $R = \QQ[x_1, \dots, x_8]$.
	
(i) Consider the edge ideal
$$I = (x_1x_2, x_2x_3, x_1x_5, x_2x_5, x_1x_6, x_2x_6, x_3x_6, x_5x_6, x_4x_7, x_5x_7, x_4x_8, x_7x_8) \subseteq R.$$
Direct computation with Macaulay2 shows that $\overline{I^3} \neq I^3$, while Corollary \ref{cor.graphcontainment}(i) gives $\overline{I^4} \subseteq I^3$.

(ii) Consider the edge ideal
$$I = (x_1x_4, x_2x_6, x_2x_7, x_3x_7, x_5x_7, x_6x_7, x_1x_8, x_2x_8, x_5x_8, x_6x_8, x_7x_8) \subseteq R.$$
Direct computation with Macaulay2 shows that $I^{(2)} \not\subseteq \overline{I^2}$ and $I^{(5)} \not\subseteq \overline{I^4}$, while Corollary \ref{cor.graphcontainment}(ii) gives $I^{(3)} \subseteq \overline{I^2}$ and $I^{(6)} \subseteq \overline{I^4}$.

(iii) Consider the edge ideal
\begin{align*}
I = (& x_1x_2, x_1x_3, x_2x_3, x_1x_4, x_2x_4, x_3x_4, x_1x_5, x_2x_5, x_3x_5, x_4x_5 \\
& x_2x_6, x_3x_6, x_5x_6, x_4x_7, x_6x_7) \subseteq R.
\end{align*}
Direct computation with Macaulay2 shows that $I^{(2)} \not\subseteq I^2$ and $I^{(4)} \not\subseteq I^3$, while Corollary \ref{cor.graphcontainment}(iii) gives $I^{(3)} \subseteq I^2$ and $I^{(5)} \subseteq I^3$.
\end{example}

As a corollary of Theorem \ref{thm.containments}(iii), we further obtain a bound for the resurgence number of squarefree monomial ideals.
Recall that for an arbitrary homogeneous ideal $I$, the {\em resurgence number} of $I$ is defined to be
$$\rho(I) = \sup\Big\{\frac{h}{k} ~\Big|~ I^{(h)} \not\subseteq I^k\Big\}.$$
This notion was due to Harbourne and Bocci \cite{BH}. Instead of $\rho(I)$, we propose to study the following closely related invariant:
$$\rho_{\inf}(I) = \inf \Big\{\frac{h}{k} ~\Big|~ I^{(h)} \subseteq I^k\Big\},$$
which is more in line with the containments between powers of $I$ as being discussed.
It is clear that $\rho(I) \le \rho_{\inf}(I)$.

\begin{corollary} \label{HB}
Let $I$ a squarefree monomial ideal. Then
$$\rho_{\inf}(I) \le d(I).$$
\end{corollary}

\begin{proof} It follows from Theorem \ref{thm.containments} that
$$\rho_{\inf}(I) \le \inf_{k \ge 1} \Big\{ \frac{d(I)k-d(I)+1}{k}\Big\} = \inf_{k \ge 1} \Big\{ \frac{d(I)(k-1)+1}{k} \Big\} \le d(I).$$ Thus, the inequality holds.
\end{proof}

The inequality $\rho(I) \le d(I)$ has also been discovered independently by a different method in \cite{DFMS}. Specializing to edge ideals of graphs, Corollary \ref{HB} gives us the following statement.

\begin{corollary} \label{cor.graph}
Let $I$ be the edge ideal of a graph. Then
$$\rho(I) \le 2.$$
\end{corollary}

\begin{remark} Let $G$ be a graph and let $I = I(G)$. Let $\chi_f(G)$ denote the fractional chromatic number of $G$ (see \cite{SU} for more details on fractional chromatic numbers of graphs). Then, it follows from \cite[Theorem 1.2.1]{BH} and \cite[Theorem 4.6]{B.etal} that\footnote{The authors thank Adam Van Tuyl for pointing them to this inequality.}
\begin{align*}
\rho(I) \ge \dfrac{2(\chi_f(G)-1)}{\chi_f(G)}.
\end{align*}
Thus, by taking graphs with large fractional chromatic numbers, we can make $\rho(I)$ to be arbitrarily close to 2. That is, the bound for $\rho(I)$ in Corollary \ref{cor.graph} and, hence, Corollary \ref{HB} is sharp.
\end{remark}

\begin{question}
Are there similar containments as those in Theorem \ref{thm.containments} (involving $d(I)$) for an arbitrary homogeneous radical (or prime) ideals?
\end{question}

It is interesting to note that containments between powers of ideals have been studied usually from a different angle, involving different invariants, for example, the minimal number of generators \cite{Huneke, LS, LT} or the maximal height of the minimal primes \cite{ELS,HH, MaS}, but not the maximal generating degree as in Theorem \ref{thm.containments}.


\section{From containments between ideals to gap estimates} \label{sec.containment2gap}

This section is a continuation of the previous section.
Making use of known containments between powers of a monomial ideal, we derive bounds for the integrality gap of certain linear programming problems, and estimate the gap between the matching and covering numbers of hypergraphs. Based on the equivalences given in Lemma \ref{lem.equiv}, we also present an equivalent algebraic reformulation for Ryser's conjecture, a long standing conjecture in hypergraph theory.

Let $M$ be an $n \times m$ matrix of non-negative integers and $\a \in \NN^n$. The integer programming problem \par
\hspace{0.8cm} maximize $\y \cdot \1^m$, \par
\hspace{0.8cm} subject to $M \cdot \y \le \a,\ \y \in \NN^m$\par
\noindent is called the {\em packing problem} in combinatorial optimization. Recall that the optimal solution of this integer programming problem and its relaxation to $\y \in \RR_{\ge 0}^n$ are denoted by $\nu_\a(M)$ and $\nu^*_\a(M)$.

The real optimal solution $\nu^*_\a(M)$ can be computed quite easily by tools from Linear Programming. On the other hand, the computation of $\nu_\a(M)$ is an {\bf NP}-hard problem.
The more intriguing question is how far $\nu_\a(M)$ differs from $\nu^*_\a(M)$.
Using the celebrated Brian\c{c}on-Skoda theorem in algebra we give the following estimate which appears not yet known in combinatorics.

\begin{theorem} \label{BS}
Let $M$ be an $n \times m$ matrix of non-negative integers. Then
$$\nu^*_\a(M) < \nu_\a(M) + \min\{m,n\}.$$
\end{theorem}

\begin{proof}
Let $\a_1, \dots, \a_m$ be the columns of $M$. Let $I$ be the monomial ideal generated by $\{x^{\a_1}, \dots, x^{\a_m}\}$. Then $\nu_\a(M) = \nu_\a(I)$ and $\nu^*_\a(M) = \nu^*_\a(I)$.
By the Brian\c{c}on-Skoda theorem (cf. \cite{Huneke, LS, LT}), we have for all $k \ge 1$,
$$\overline{I^{k+\min\{m,n\}-1}} \subseteq I^k.$$
Let $f(k) := k + \min\{m,n\}-1$.
Applying Lemma \ref{lem.equiv}(ii) to the ideals $\overline{I^k}$, $I^k$, the functions $\nu^*_\a(M)$, $\nu_\a(M)$, and $f(k)$, we obtain
$$\nu^*_\a(I) < f(\nu_\a(I)+1) = \nu_\a(I)+\min\{m,n\},$$
which proves the assertion.
\end{proof}

Another gap estimate is related to a conjecture of Harbourne (see, for example, \cite{Ba}),
which asks whether the containment $I^{(hk-h+1)} \subseteq I^k$ holds for every proper homogeneous ideal $I$ and $k \ge 1$, where $h$ denotes the maximal height of an associated prime of $I$. This conjecture is inspired by the formula
$I^{(hk)} \subseteq I^k$, which was discovered by Ein, Lazarsfeld, and Smith \cite{ELS}, Hochster and Huneke \cite{HH}, Ma and Schwede \cite{MaS}.
The conjecture of Harbourne has an affirmative answer when $I$ is a squarefree monomial ideal (see \cite{Ba, CEHH}).
Making use of this result, we deduce the following estimate for the gap between $\tau_\a(M)$ and $\nu_\a(M)$.

\begin{theorem} \label{Ha}
Let $M$ be the incidence matrix of a simple hypergraph $\H$.
Let $h$ be the maximal size of a minimal cover of $\H$.
Then for all $\a \in \NN^n$,
$$\tau_\a(M) \le h\nu_\a(M).$$
\end{theorem}

\begin{proof}
Let $I$ be the edge ideal of $\H$. Then for all $\a \in \NN^n$, we have
$\tau_\a(M) = \tau_\a(I)$ and $\nu_\a(M) = \nu_\a(I).$
It follows from \cite[Example 8.4.5]{Ba} (see also \cite[Corollary 4.4]{CEHH}) that
$$I^{(hk-h+1)} \subseteq I^k$$
for all $k \ge 1$. Set $f(k) := h(k-1)+1.$
Applying Lemma \ref{lem.equiv}(ii) to the ideals $I^{(k)}$, $I^k$, the functions $\tau_\a(I)$, $\nu_\a(I)$, and $f(k)$, we obtain
$$\tau_\a(I) < f(\nu_\a(I)+1) = h\nu_\a(I)+1.$$
Since $\tau_\a(I)$ and $\nu_\a(I)$ are integers, this implies $\tau_\a(I) \le h\nu_\a(I).$ The conclusion follows.
\end{proof}

In Theorem \ref{Ha}, we cannot replace $h$ by the minimal size of a minimal cover of $\H$, which is $\tau(\H)$.
Algebraically, this means that the formula
$$I^{(\height(I)k - \height(I) +1)} \subseteq I^k$$
does not hold for an arbitrary squarefree monomial ideal $I$ and all $k \ge 1$.

\begin{example}
Let $\G$ be the hypergraph whose edges are $\{1,2\}$ and all 5-subsets of $[1,8]$ not containing $\{1,2\}$.
Let $\H$ be the hypergraph whose edges are subsets of $[1,8]$ of the form $\{1,2,i,j\}$, $\{1,i,j,t\}$ and $\{2,i,j,t\}$, where $3 \le i,j,t \le 8$ are different numbers. It is easy to check that edges in $\H$ are the minimal covers of $\G$. That is, $\H = \G^\vee$. Thus, $\G = \H^\vee$. In particular, $\tau(\H) = 2$. \par
Let $I$ be the edge ideal of $\H$ in $K[x_1,\dots,x_8]$. Then
$$I^{(k)} = \bigcap_{F \in \G} P_F^k.$$
It is easy to see that $f := x_1^3x_2^2x_3 \dots x_8 \in I^{(5)}$. Since $\deg(f) = 11$, $f \not \in I^3$ because $I$ is generated by monomials of degree 4. Therefore, $I^{(2k-2+1)} = I^{(2k-1)}\not\subseteq I^k$ for $k = 3$.
By Lemma \ref{lem.equiv}(ii), we conclude that the inequality $\tau_\a(M) \le 2\nu_\a(M)$ 
does not hold for all $\a \in \NN^8$, where $M$ is the incident matrix of $\H$.
\end{example}

\begin{remark}
If $M$ is the incidence matrix of a hypergraph $\H$ and if $\a = \1^n$ then the bounds in Theorems \ref{BS} and \ref{Ha} are trivial.
In this case, we have
\begin{align*}
\nu^*_\a(\H) \le \tau_\a(\H) & = \tau(\H) \le \min\{m,n\},\\
\tau_\a(\H) & = \tau(\H) \le h.
\end{align*}
Hence, these results are interesting only for more general $\a \in \NN^n$.
\end{remark}

In the study of parallelization of hypergraphs, it is often of interest to ask the following question: given a hypergraph $\G$, which hypergraph $\H$ has the smallest number of vertices such that $\G = \H^\a$ for some positive integral vector $\a$? In investigating this question, the following notions prove to be of importance.

Two vertices $u$ and $v$ of $\G$ are said to be \emph{clones} if $u,v$ are not contained in any edge of $\G$, and $F$ is an edge in $\G$ containing $u$ if and only if $F - u+v$ is an edge in $\G$. The terminology \emph{clone} is adapted from \cite{MNR}.
In particular, if $\G$ is a graph, then $u, v$ are \emph{clones} if and only if $u,v$ are \emph{twins}, i.e., they share the same open neighborhood.
It follows from the definition that if $\G = \H^\a$ for a hypergraph $\H$ and a positive integral vector $\a = (\alpha_1, \dots, \alpha_n)$, then for each $i \in \supp(\a)$ with $\alpha_i \ge 2$, the vertices $\{i_1,\dots,i_{\alpha_i}\}$ of $\G$ are pairwise clones.

Now, for each vertex $u \in \G$ we denotes by $[u]$ the class of the clones of $u$.  Let $\H$ denote the hypergraph whose vertices are the clone classes and whose edges are sets of the form $\{[u_1],\dots,[u_s]\}$ with $\{u_1,\dots,u_s\}$ being an edge of $\G$. Assume that $\G$ has $n$ different clone classes whose cardinality are $\alpha_1,\dots,\alpha_n$. It is easy to see that $\G = \H^\a$ for $\a = (\alpha_1,\dots,\alpha_n)$. It can be shown that $\H$ is a hypergraph with the smallest number of vertices such that $\G = \H^\a$ for some positive integral vector $\a$.
For simplicity, we call $\H$ the \emph{reduced clone-free hypergraph} of $\G$.

Using the reduced clone-free hypergraph we can improve the bound of Theorem \ref{Ha} as follows.

\begin{theorem} \label{thm.tauhnu}
Let $\G$ be an arbitrary hypergraph. Let $h^*$ denote the maximum cardinality of minimal covers of the reduced clone-free hypergraph of $\G$. Then
$$\tau(\G) \le h^*\nu(\G).$$
\end{theorem}

\begin{proof}
Let $\H$ be the reduced clone-free hypergraph of $\G$, and suppose that $\H$ contains $n$ vertices.
Let $\a \in \NN^n$ be such that $\G = \H^\a$. Let $M$ be the incidence matrix of $\H$.
By Proposition \ref{parallel-2},
$\tau_\a(M) = \tau(\G)$ and $\nu_\a(M) = \nu(\G).$
Applying Theorem \ref{Ha} to $\H$, we obtain
$$\tau_\a(M) \le h^*\nu_\a(M).$$
Therefore, $\tau(\G) \le h^*\nu(\G)$.
\end{proof}

By Lemma \ref{parallel-1}(ii), the maximum cardinality of minimal covers of the reduced clone-free hypergraph of $\G$ is less than the maximum cardinality of minimal covers of $\G$. In fact, the difference between these invariants could be made arbitrarily large as seen in the following example. This exhibits the fact that the conclusion of Theorem \ref{thm.tauhnu}, in practice, is significantly stronger than that of Theorem \ref{Ha}.

\begin{example}
Let $\G = K_{1,p}$ be the complete bipartite graph on $\{x; y_1, \dots, y_p\}$. Clearly, $\{y_1, \dots, y_p\}$ is a minimal vertex cover of $G$. Thus, the invariant $h$ in Theorem \ref{Ha} for this example is $p$. On the other hand, let $\H$ be the graph consisting of a single edge $\{x,y\}$, and let $\a = (1,p) \in \NN^2$. Then $\G = \H^\a$ and, so, the invariant $h^*$ in Theorem \ref{thm.tauhnu} for this example is 1.
\end{example}

We now turn our attention to a long standing open conjecture in hypergraph theory, the Ryser's conjecture.
Recall that a hypergraph $\H$ is said to be \emph{$r$-partite} if there is a partition of its vertex set into $r$ parts such that no edge in $\H$ contains two vertices from the same part.

\begin{conjecture}[Ryser] \label{conj.Ryser}
Let $H$ be an $r$-partite hypergraph of rank $\le r$. Then
$$\tau(H) \le (r-1)\nu(H).$$
\end{conjecture}

This conjecture is often formulated for $r$-partite hypergraphs which are $r$-uniform,
i.e., all edges are of the same size $r$.
In fact, we can always add new and distinct vertices to edges of an $r$-partite hypergraph of rank $\le r$ to get an $r$-uniform $r$-partite hypergraph with the same matching and covering numbers.

In connection to Ryser's conjecture, we shall make the following conjecture on the containment between symbolic and ordinary powers of squarefree monomial ideals.

\begin{conjecture} \label{conj.containment}
Let $I$ be the edge ideal of an $r$-partite hypergraph of rank $\le r$. Then, for all $k \in \NN$, we have
$$I^{((r-1)(k-1)+1)} \subseteq I^k.$$
\end{conjecture}

\begin{theorem}  \label{equiv}
Ryser's conjecture is equivalent to Conjecture \ref{conj.containment}.
\end{theorem}

\begin{proof}
Observe that if $\H$ is an $r$-partite hypergraph of rank at most $r$, then so is $\H^\a$ for any $\a \in \NN^n$.
By Proposition \ref{parallel-2}, we have $\tau_\a(I) = \tau(\H^\a) \text{ and } \nu_\a(I) = \nu(\H^\a)$,
where $I$ is the edge ideal of $\H$.
Therefore, Ryser's conjecture can be rewritten as
$$\tau_\a(I) < (r-1)\nu_\a(I)+1$$
for all $\a \in \NN^n$, where $I$ is the edge ideal of a $r$-partite hypergraph of rank at most $r$.
Set $f(k) = (r-1)(k-1)+1$. Then
$$f(\nu(\a)+1) = (r-1)\nu(\a)+1.$$
Applying Lemma \ref{lem.equiv}(ii) to the ideals $I^{(k)}$ and $I^k$, together with the functions $\tau(\a)$, $\nu(\a)$ and $f(k)$, we immediately obtain the assertion.
\end{proof}

If we replace $I^{((r-1)(k-1)+1)}$ by $\overline{I^{((r-1)(k-1)+1}}$ or $I^k$ by $\overline{I^k}$ in Conjecture \ref{conj.containment} then it has a positive answer.
This follows from the following result. Note that
$$I^{((r-1)(k-1)+1)} \subseteq I^{(\lceil \frac{1}{2}r(k-1)\rceil+1)}$$
for all $r \ge 2$, $k \ge 1$.


\begin{theorem}\label{symbolicVSintegral3}
Let $I$ be the edge ideal of a simple $r$-partite hypergraph.
Then, for any $k \in \NN$, we have \par
{\rm (i)} $\overline{I^{(r-1)(k-1)+1}} \subseteq I^k;$ \par
{\rm (ii)} $I^{(\lceil \frac{1}{2}r(k-1)\rceil+1)} \subseteq \overline{I^k}.$
\end{theorem}

\begin{proof}
(i)  Due to an unpublished result of Gy\'arf\'as \cite[Corollary 5]{F} we have
$$\nu^*(\H) \le (r-1)\nu(\H)$$
for any $r$-partite hypergraph $\H$. By Proposition \ref{parallel-2}, this implies that
$$\nu_\a^*(I) < f(\nu_\a(I)+1),$$
where $f(k) := (r-1)(k-1)+1$. Applying Lemma \ref{lem.equiv}(ii) to the ideals $\overline{I^k}$ and $I^k$ with the functions $\nu^*(\a)$ and $\nu(\a)$ and this function $f(k)$, we obtain the assertion as in the proof of Theorem \ref{equiv}. \par
(ii) By a result of Lovasz  in \cite{Lo1, Lo2}, we have
$$\tau(\H) \le \dfrac{1}{2}r\nu^*(\H)$$
for any $r$-partite hypergraph $\H$. Thus, the assertion follows from Lemma \ref{lem.equiv}(ii) similarly as above.
\end{proof}


\section{Equality between powers of monomial ideals} \label{sec.equality}

We have dealt with the containments between the ideals $I^k, \overline{I^k}, I^{(k)}$.
In this section we will investigate the equality between these ideals.
We will use the following simple observation.

\begin{lemma} \label{equality}
Let $\{I_k\}_{k \ge 1}$ and $\{J_k\}_{k \ge 1}$ be two filtrations of monomial ideals in $R$.
Assume that there are functions $\mu$ and $\rho$ from $\NN^n$ to $\RR_+$ such that, for any $\a \in \NN^n$ and $k \ge 1$, \par
\begin{itemize}
\item  $x^\a \in I_k$ if and only if $\mu(\a) \ge k$,\par
\item $x^\a \in J_k$ if and only if $\rho(\a) \ge k$.
\end{itemize}
Then $I_k = J_k$ for all $k \ge 1$ if and only if $\lfloor \mu(\a) \rfloor = \lfloor \rho(\a) \rfloor$ for all $\a \in \NN^n$.
\end{lemma}

\begin{proof}
We have $I_k = J_k$  if and only if $\mu(\a) \ge k$ is equivalent to $\mu(\a) \ge k$ for all $k \ge 1$. This equivalence just means exactly that $\lfloor \mu(\a) \rfloor = \lfloor \rho(\a) \rfloor$ for all $\a \in \NN^n$.
\end{proof}

An ideal $I$ is called {\em normal} if $\overline{I^k} = I^k$ for all $k \ge 1$.
If $I$ is a monomial ideal then we have the following effective criterion for this property.

\begin{theorem} \label{normal} \cite[Corollary 4.5]{DV}, \cite[Theorem 3.2]{T2}
Let $I$ be a monomial ideal in $R$. Then 
$I$ is normal if and only if $\nu_\a(I) = \lfloor \nu_\a^*(I) \rfloor$ for all $\a \in \NN^{n}$.
\end{theorem}

\begin{proof}
The statement immediately follows from the membership criteria for $I^k$ and $\overline{I^k}$ in Proposition \ref{ordinary} and Lemma \ref{equality}.
\end{proof}

In combinatorics, a matrix $M$ of non-negative integers with $n$ rows is said to have the {\em integer round-down property}
if $\nu_\a(M) = \lfloor \nu_\a^*(M) \rfloor$ for all $\a \in \NN^{n}$ \cite{BT1}.
Several classes of matrices have been shown to have this property.
Therefore, Theorem \ref{normal} can be used to find new classes of normal ideals. \par

Our focus is on the case where $M$ is the incident matrix of a hypergraph $\H$. Let $I$ be the edge ideal of $\H$.
Then, for all $k \ge 1$, we have
$$I^k \subseteq \overline{I^k}  \subseteq I^{(k)}.$$
On the other hand, for all $\a \in \NN^n$, we also have
$$\nu_\a(M) \le \nu^*_\a(M) = \tau_\a^*(M) \le \tau_\a(M).$$

Following the terminology in hypergraph theory \cite{Du, Sch}, we say that
\begin{itemize}
\item $\H$ has the {\em integer round-down property} if $\nu_\a(M)= \lfloor \nu^{*}_\a(M)\rfloor$ for all $\a \in \NN^{n}$;
\item $\H$ is called {\em Fulkersonian} (or {\em ideal}) if
$\tau_\a(M) =\tau^{*}_\a(M)$ for all $\a \in \NN^{n}$;
\item $\H$ is called {\em Mengerian} (or has the {\em max-flow min-cut property}) if
$\nu_\a(M) = \tau_\a(M)$ for all $\a \in \NN^{n}$.
\end{itemize}

Note that $\tau_\a(M) = \tau^{*}_\a(M)$ is equivalent to $\tau_\a(M) = \lfloor \tau^{*}_\a(M)\rfloor$ because
$$\lfloor \tau^{*}_\a(M)\rfloor \le \tau^{*}_\a(M) \le \tau_\a(M).$$

\begin{remark}
We do not need to check the above equalities for all $\a \in \NN^n$.
By \cite[Corollary 2.3]{BT2}, there is a well-determined vector $\b \in \NN^n$ (depending on $M$) such that $\H$ has the integer round-down property or $\H$ is Fulkersonian if and only if $\nu_\a(M) = \lfloor \nu_\a^*(M) \rfloor$ or $\tau_\a(M) = \tau^{*}_\a(M)$, respectively,
for all $\a \le \b$. Since $\H$ is Mengerian if and only if $\H$ has the integer round-down property and $\H$ is Fulkersonian, we can also check the Mengerian property in a finite number of steps.
\end{remark}

By Lemma \ref{equality}, the membership criteria for $I^k$, $\overline{I^k}$ and $I^{(k)}$ immediately yield the following results. 

\begin{theorem} \label{correspondence}
Let $I$ be the edge ideal of a hypergraph $\H$. Then 
\begin{enumerate}
\item[\rm (i)] $I^k = \overline{I^k}$ for all $k \ge 1$ if and only if $\H$ has the integer round-down property \cite[Corollary 4.5]{DV}, \cite[Theorem 3.7(1)]{T2};
\item[\rm (ii)] $\overline{I^k} = I^{(k)}$ for all $k \ge 1$ if and only if $\H$ is a Fulkersonian hypergraph  \cite[Theorem 3.1]{T1};
\item[\rm (iii)] $I^k = I^{(k)} $ for all $k \ge 1$ if and only if $\H$ is a Mengerian hypergraph \cite[Corollary 3.5]{GVV}, \cite[Corollary 1.6]{HHTZ}.
\end{enumerate}
\end{theorem}

\begin{remark}
It might be tempting to state that $I^k = \overline{I^k}$ for all $k \ge 1$ if and only if $\nu_\a(I)= \nu^{*}_\a(I)$ for all $\a \in \NN^{n}$.
However, the latter condition is satisfied if and only if $\H$ is a Mengerian hypergraph \cite[Theorem 79.2]{Sch}; that is, when $I^k = I^{(k)}$ for all $k \ge 1$.
\end{remark}

Let $V_1, V_2$ be two arbitrary disjoint subsets of the vertex set $[1,n]$ of $\H$.
We define a hypergraph $\G$ on the set of vertices $[1,n] \setminus (V_1 \cup V_2)$
whose edges are the subsets of $V$ of the form $F \setminus V_1$,
where $F \in \H$ and $F \cap V_2 = \emptyset$. We call $\G$ a {\it minor} of  $\H$. \par

Recall that a hypergraph  $\H$ is \emph{K\"onig} if $\nu(\H)=  \tau(\H)$ \cite{Du}.
If all minors of $\H$ are K\"onig then $\H$ is said to have the \emph{packing property} \cite{Sch}.

It is easy to see that minors of $\H$ are exactly the parallelizations $\H^\a$ with $\a \in \{0,1\}^n$.
Thus, a Mengerian hypergraph has packing property.
The converse was a conjecture raised in 1993 by Conforti-Cornu\'ejols \cite{CC}.

\begin{conjecture}[Conforti-Cornu\'ejols]
A hypergraph with packing property is Mengerian.
\end{conjecture}

By definition, $\H$ is Mengerian if and only if $\H$ is Fulkersonian and has the integer round-down property.
It is known that $\H$ is a Fulkerson hypergraph if $\H$ has packing property \cite[Corollary 78.4b]{Sch}.
Therefore, the Conforti-Cornu\'ejols conjecture is equivalent to the following conjecture.

\begin{conjecture} \label{CC'}
A hypergraph with packing property has the integer round-down property.
\end{conjecture}

It is easy to see that $\tau(\H) = \text{ht}(I)$ and $\nu(\H) = \mongrade(I)$,
where $\mongrade(I)$ denotes the maximal length of a regular sequence of monomials in $I$.
Therefore, $\H$ is K\"onig if and only if $\mongrade(I) = \height(I).$

Let $\G$ be a minor of $\H$ with respect to two disjoint subsets $V_1, V_2 \subseteq [1,n]$.
By the definition of minor, the edge ideal $J$ of $\G$ is obtained from $I$ by setting
$x_i = 1$ for $i \in V_1$ and $x_i = 0$ for $i \in V_2$.
That means $J$ is the ideal of the monomials in the polynomial ring $K[x_i|\ i \not\in V_1 \cup V_2]$ generated from those of $I$ by setting $x_i = 1$ for $i \in V_1$ and $x_i = 0$ for $i \in V_2$.

It is now clear that Conjecture \ref{CC'} can be translated in algebraic terms as follows.

\begin{conjecture} \label{CC} \cite[p. 4]{T2}
Let $I$ be a squarefree monomial ideal such that
$$\mongrade(J) = \height(J)$$
for all monomial ideals $J$ obtained from $I$ by setting some variables equal to 0,1.
Then $I$ is a normal ideal.
\end{conjecture}

Other algebraic interpretations and variants of the Conforti-Cornu\'ejols conjecture can be found, e.g., in the surveys \cite{Da, FHM}.


\section{Maximal generating degree of symbolic powers} \label{sec.generatingdegree}

In this section we will use techniques developed in preceding sections to investigate the following problem.

\begin{problem} \label{Huneke}
Let $I$ be a homogeneous radical ideal. Is $d(I^{(k)}) \le kd(I)$ for all $k \ge 1$?
\end{problem}

This problem was originally raised for prime ideals by Huneke \cite[Question 0.5]{Hu}.
We shall see that for squarefree monomial ideals, Problem \ref{Huneke} can be reduced to
a problem on the relationship between $n$ and $d(I)$.

\begin{lemma} \label{generator-1}
Let $I$ be an arbitrary squarefree monomial ideal, and let $\a \in \NN^n$ and $k \ge 1$. If
$x^\a$ is a minimal generator of $I^{(k)}$ then $\tau_\a(I) = k$.
\end{lemma}

\begin{proof}
It is clear that $x^\a$ is a minimal generator of $I^{(k)}$ if and only if
$x^\a \in I^{(k)}$ and $x^{\a-\e_i} \not\in I^{(k)}$ for all $i \in \supp(\a)$,
where $\e_i$ denotes the $i$-th unit vector in $\NN^n$.
By Proposition \ref{symbolic}, this means that $\tau_\a(I) \ge k$ and $\tau_{\a-\e_i}(I) \le k-1$ for all $i \in \supp(\a)$.

By Lemma \ref{blocker}, we have
$$\tau_\a(I) = \min\{\a \cdot \e_F|\ F \in \H^\vee\}.$$
Thus, if $\tau_\a(I) > k$ then for all $F \in \H^\vee$, $\a \cdot \e_F \ge k+1$. This implies that for all $i \in \supp(\a)$ and all $F \in \H^\vee$, we have $(\a - \e_i) \cdot \e_F \ge k$, i.e., $\tau_{\a-\e_i}(I) \ge k$, a contradiction.
\end{proof}

\begin{lemma} \label{generator-2}
Let $I$ be the edge ideal of a hypergraph $\H$ and $\a \in \NN^n$.
Let $I_\a$ denote the edge ideal of the parallelization $\H^\a$ in the polynomial ring
$$S = K[x_{ij}|\ i = 1,\dots,n,\ j = 1,\dots,\alpha_i].$$
Assume that $x^\a$ is a minimal generator of $I^{(k)}$, $k \ge 1$. Then \par
{\rm (i)} $\height(I_\a) = k$,\par
{\rm (ii)} Every variable $x_{ij}$ belongs to at least a minimal prime of $I_\a$ of height $k$.
\end{lemma}

\begin{proof}
(i) By Lemma \ref{generator-1}, we have $\tau_\a(I) = k$.
Hence, $\tau(\H^\a) = k$ by Proposition \ref{parallel-2}(iv).
This implies $\height(I_\a) = \tau(\H^\a) = k$.
\par

(ii) Assume that there is a variable $x_{ij}$ that does not belong to any minimal prime of $I_\a$ of height $k$.
Then $i \in \supp(\a)$ and $(i,j) \not\in E$ for all $E \in (\H^\a)^\vee$. By Lemma \ref{parallel-1},
$(\H^\a)^\vee = \{p^{-1}(F)|\ F \in \H^\vee\}.$
Hence, $i \not\in F$ for all $F \in \H^\vee$.
This implies that $\e_i \cdot  \e_F = 0$ for all $F \in \H^\vee$.
Therefore, $(\a-\e_i) \cdot \e_F = \a \cdot \e_F$. By Lemma \ref{blocker}, we now have
$$\tau_{\a-\e_i}(I) = \tau_\a(I) = k.$$
Hence, $x^{\a - \e_i} \in I^{(k)}$ by Proposition \ref{symbolic}.
It follows that $x^\a$ is not a minimal generator of $I^{(k)}$, a contradiction.
\end{proof}

\begin{proposition} \label{equi}
Let $f : \NN \to \NN$ be a numerical non-decreasing function.
The following conditions are equivalent:\par
{\rm (i)} For any squarefree monomial ideal $I$, $d(I^{(k)}) \le kf(d(I))$ for all $k \ge 1$.   \par
{\rm (ii)} $n \le \height(I)f(d(I))$ for every squarefree monomial ideal $I$ in $n$ variables such that every variable appears in at least a minimal prime of $I$ with minimal height.
\end{proposition}

\begin{proof}
Assume that (i) is satisfied.
Let $I$ be a squarefree monomial ideal in $n$ variables such that every variable appears in at least a minimal prime of $I$ with minimal height. Let $\H$ be a hypergraph such that $I$ is the edge ideal.
We know that
$$I^{(k)} = \bigcap_{F \in \H^\vee} P_F^k.$$
Let $k = \height(I)$. Since every minimal prime $P_F$ of $I$ is generated by at least $k$ variables,
$x_1 \cdots x_n \in P^k$. Hence, $x_1 \cdots x_n \in I^{(k)}$.
Since every variable $x_i$ appears in at least a minimal prime of $I$ generated by $k$ variables,
$(x_1 \cdots x_n)/x_i \not\in P_F^k$ for some $F \in \H^\vee$.  Hence, $(x_1 \cdots x_n)/x_i \not\in I^{(k)}$.
It follows that $x_1 \cdots x_n$ is a minimal generator of $I^{(k)}$.
Hence, $n \le d(I^{(k)}).$
Since $d(I^{(k)}) \le kf(d(I))$, we obtain $n \le kf(d(I)) = \height(I)f(d(I)).$ \par

Assume that (ii) is satisfied.
Let $I$ be an arbitrary squarefree monomial ideal.
Let $\H$ be a hypergraph such that $I$ is the edge ideal of $\H$.
Let $x^\a$ be an arbitrary minimal generator of $I^{(k)}$ with $\deg x^\a = d(I^{(k)})$.
Let $I_\a$ denote the edge ideal of the parallelization $\H^\a$.
Then $d(I_\a) \le d(I)$ by the definition of $\H^\a$, and $\height(I_\a) = k$ by Lemma \ref{generator-2}(i).
Since $\H^\a$ has $\alpha_1 + \cdots + \alpha_n$ variables, $I_\a$ lies in a polynomial ring $S$ in $\deg x^\a$ variables.
By Lemma \ref{generator-2}(ii), every variable of $S$ belongs to at least a minimal prime of $I_\a$ of minimal height.
Therefore, $\deg x^\a \le \height(I_\a)f(d(I_\a))$ and, hence, $d(I^{(k)}) \le kf(d(I))$.
\end{proof}

If Problem \ref{Huneke} has a positive answer, we would have $n \le \height(I)d(I)$ for every squarefree monomial ideal $I$ in $n$ variables such that every variable appears in at least a minimal prime of $I$ with minimal height.
To find a counter-example to Problem \ref{Huneke}, we only need to look for such a squarefree monomial ideal $I$ with small $\height(I)$ and $d(I)$ in a polynomial ring with a large number of variables. \par

In fact, there are squarefree monomial ideals with $\height(I) = 2$ such that $n - 2d(I)$ is arbitrarily large. From this it follows that $d(I^{(2)}) - 2d(I)$ can be arbitrarily large, too.

\begin{example} \label{counter}
Let $m \ge 2$ be an arbitrary integer. Let $\H$ be the graph on $n = 3(m+1)$ vertices which consists of a triangle $T$ and $3m$ leaves, where every vertex of $T$ is has exactly $m$ leaves; see Figure 2. 

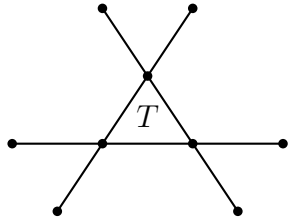
\begin{figure}[ht!]
\begin{tikzpicture}[scale=0.6]

\draw [thick] (0,0) coordinate (a) -- (2,0) coordinate (b) ;
\draw [thick] (0,0) coordinate (a) -- (1,1.5) coordinate (c) ;
\draw [thick] (2,0) coordinate (b) -- (1,1.5) coordinate (c) ;
\draw [thick] (0,0) coordinate (a) -- (-2,0) coordinate (d);
\draw [thick] (0,0) coordinate (a) -- (-1,-1.5) coordinate (e);
\draw [thick] (1,1.5) coordinate (c) -- (0,3) coordinate (f);
\draw [thick] (1,1.5) coordinate (c) -- (2,3) coordinate (g);
\draw [thick] (2,0) coordinate (b) -- (4,0) coordinate (h);
\draw [thick] (2,0) coordinate (b) -- (3,-1.5) coordinate (i);

\draw (1,0.6) node{$T$};

\fill (a) circle (3pt);
  \fill (b) circle (3pt);
  \fill (c) circle (3pt);
  \fill (d) circle (3pt);
  \fill (e) circle (3pt);
  \fill (f) circle (3pt);
  \fill (g) circle (3pt);
  \fill (h) circle (3pt);
  \fill (i) circle (3pt);

\end{tikzpicture}
\caption{A hypergraph consisting of a triangle and 6 leaves}
\end{figure}

Let $I = \cap_{{i,j} \in \H} (x_i,x_j)$. Then $\height(I) = 2$.
To compute the minimal generators of $I$ we have to find the minimal covers of $\H$.
A minimal cover of $H$ must contain at least 2 vertices of $T$.
Using this fact one can see that a minimal cover of $\H$ is either the set of the 3 vertices of $T$ or a set which consists of 2 vertices of $T$ and the $m$ vertices adjacent to the remaining vertex of $T$.
From this it follows that $d(I) = m+2$. Hence,
$$n = 3(m+1) > 2(m+2) = \height(I)d(I).$$
In particular, $d(I^{(2)}) - 2d(I) \ge n - 2d(I) = m-1$ can be arbitrarily large.
\end{example}

Counter-examples to Problem \ref{Huneke} were found in Asgharzadeh  \cite{As} (with an attribute to Hop D. Nguyen).
In these examples, $I$ is a non-monomial radical ideal with $d(I^{(2)}) - 2d(I) = 1$.

Due to Example \ref{counter}, we modify Problem \ref{Huneke} as follows.

\begin{problem} \label{prob.new}
Let $I$ be a squarefree monomial ideal.
Does there exist a function $f: \NN \longrightarrow \NN$ such that
$d(I^{(k)}) \le kf(d(I))$ for all $k \ge 1$?
\end{problem}

If $I$ is the edge ideal of a hypergraph, and if the involved invariants do not increase when passing to the edge ideals of parallelizations of the given hypergraph, then Problem \ref{prob.new} is amount to the question of whether $n \le \height(I)f(d(I))$ as in Proposition \ref{equi}.


\end{document}